\documentclass[11pt]{amsart}
\usepackage{amsmath, amsthm, mathabx,amscd, amsfonts, amssymb, hyperref, mathrsfs, textcomp, epsfig, a4wide, graphicx, IEEEtrantools, tikz, verbatim, xypic,enumerate}

\usetikzlibrary{matrix}

\newtheorem{thm}{Theorem}
\newtheorem{prop}{Proposition}
\newtheorem{lemma}{Lemma}

\theoremstyle{definition}
\newtheorem{defn}{Definition}

\theoremstyle{remark}

     \RequirePackage{rotating}                   
    \def\HSt{%
       \setbox0=\hbox{$\widehat{\mathit{HS}}$}
       \setbox1=\hbox{$\mathit{HS}$}
       \dimen0=1.1\ht0
       \advance\dimen0 by 1.17\ht1
       \smash{\mskip2mu\raise\dimen0\rlap{%
          \begin{turn}{180}
              {$\widehat{\phantom{\mathit{HS}}}$}
           \end{turn}} \mskip-2mu    
                \mathit{HS}
    }{\vphantom{\widehat{\mathit{HS}}}}{}}

     \RequirePackage{rotating}                   
    \def\HMt{%
       \setbox0=\hbox{$\widehat{\mathit{HM}}$}
       \setbox1=\hbox{$\mathit{HM}$}
       \dimen0=1.1\ht0
       \advance\dimen0 by 1.17\ht1
       \smash{\mskip2mu\raise\dimen0\rlap{%
          \begin{turn}{180}
              {$\widehat{\phantom{\mathit{HM}}}$}
           \end{turn}} \mskip-2mu    
                \mathit{HM}
    }{\vphantom{\widehat{\mathit{HM}}}}{}}
    
    \newcommand{\HMb}{\overline{\mathit{HM}}}

    \newcommand{\HSb}{\overline{\mathit{HS}}}
\newcommand{\HSf}{\widehat{\mathit{HS}}}
\newcommand{\spin}{\mathfrak{s}}
\newcommand{\ztwo}{\mathbb{F}}
\newcommand{\Pin}{\mathrm{Pin}(2)}
\newcommand{\Rin}{\mathcal{R}}

\newcommand{\s}{\mathbf{s}}
\newcommand{\x}{\mathbf{x}}
\newcommand{\mt}{\text{mod }2}
\newcommand{\K}{\mathrm{K}}
\newcommand{\KO}{\mathrm{KO}}
\newcommand{\KR}{\mathrm{KR}}
\newcommand{\KSp}{\mathrm{KSp}}
\newcommand{\KQ}{\mathrm{KQ}}
\newcommand{\Op}{\mathrm{Op}}
\newcommand\numberthis{\addtocounter{equation}{1}\tag{\theequation}}

\begin{document}

\begin{abstract}
We show that the bar version of the $\Pin$-monopole Floer homology of a three-manifold $Y$ equipped with a self-conjugate spin$^c$ structure $\spin$ is determined by the triple cup product of $Y$ together with the Rokhlin invariants of the spin structures inducing $\spin$. This is a manifestation of mod $2$ index theory, and can be interpreted as a three-dimensional counterpart of Atiyah's classic results regarding spin structures on Riemann surfaces.
\end{abstract}

\title{$\mathrm{PIN}(2)$-Monopole Floer homology and the Rokhlin invariant}

\author{Francesco Lin}
\address{Department of Mathematics, Princeton University and Institute for Advanced Study} 
\email{fl4@math.princeton.edu}
\maketitle

In \cite{Lin} we introduced for each closed oriented three-manifold $(Y,\spin)$ equipped with a self-conjugate spin$^c$ structure (i.e. $\spin=\bar{\spin}$) the $\Pin$-monopole Floer homology groups
\begin{equation*}
\HSb_*(Y,\spin),\quad \HSt_*(Y,\spin),\quad \HSf_*(Y,\spin).
\end{equation*}
These are graded modules over the ring $\Rin=\ztwo[V,Q]/Q^3$, where $V$ and $Q$ have degrees respectively $-4$ and $-1$, and $\ztwo$ is the field with two elements. To define them, one exploits the
\begin{equation*}
\Pin=S^1\cup j\cdot S^1\subset\mathbb{H}
\end{equation*}
symmetry of the Seiberg-Witten equations, and in the case $b_1(Y)=0$ they are the analogue of Manolescu's invariants (\cite{Man2}) in the context of Kronheimer and Mrowka's monopole Floer homology (\cite{KM}). In particular, they can be used to provide an alternative disproof of the Triangulation conjecture. We refer the reader to \cite{Lin3} for a friendly introduction to the construction, and to \cite{Man3} for a survey on the Triangulation conjecture.
\\
\par
In the present paper, we will focus on the simplest of the three invariants, $\HSb_*(Y,\spin)$. This only involves reducible solutions and, heuristically, it computes the homology of the boundary of the moduli space of configurations. It is shown in \cite{KM} that their monopole Floer homology $\HMb_*(Y,\spin)$ is determined entirely by the triple cup product
\begin{align*}
\cup^3_Y:\Lambda^3H^1(Y;\mathbb{Z})&\rightarrow \mathbb{Z}\\
a_1\wedge a_2\wedge a_3&\mapsto \langle a_1\cup a_2\cup a_3, [Y]\rangle.
\end{align*}
This is a key step in the proof of their non-vanishing theorem (Corollary $35.1.3$ in \cite{KM}), which is in turn one of the main ingredients of Taubes' proof of the Weinstein conjecture in dimension $3$, see \cite{Tau}. In our set-up, recall that there is a natural map
\begin{equation*}
\left\{\text{spin structures}\right\}\rightarrow \left\{\text{self-conjugate spin$^c$ structures}\right\}
\end{equation*}
which is surjective and has fibers of cardinality $2^{b_1(Y)}$. This should be compared to the Bockstein sequence
\begin{equation*}
\cdots\longrightarrow H^1(Y;\mathbb{Z})\longrightarrow H^1(Y;\ztwo)\longrightarrow H^2(Y;\mathbb{Z})\stackrel{\cdot 2}{\longrightarrow}H^2(Y;\mathbb{Z})\longrightarrow\cdots,
\end{equation*}
which implies that the set of spin structures inducing $\spin$ (which we denote by $\mathrm{Spin}(\spin)$) is an affine space over $H^1(Y;\mathbb{Z})\otimes\ztwo$. We will denote an element of $\mathrm{Spin}(\spin)$ by $\s$. For each spin structure $\s$ there is the (rational) Rokhlin invariant
\begin{equation*}\mu(\s)=\sigma(W)/8\text{ mod }2
\end{equation*}
where $W$ is any four-manifold whose boundary is $Y$ and on which $\s$ extends. We will think of the collection of the Rokhlin invariants of the spin structures in $\mathrm{Spin}(\spin)$ as a map $\mu_{\spin}$ (called the Rokhlin map). With this in hand, we are ready to state our main theorem.
\begin{thm}\label{main}
The Floer homology group $\HSb_*(Y,\spin)$ is determined by the triple cup product $\cup^3_Y$ and the Rokhlin map $\mu_{\spin}$. More precisely, given $(Y_0,\spin_0)$ and $(Y_1,\spin_1)$, suppose that:
\begin{itemize}
\item there exist (respectively linear and affine) isomorphisms
\begin{align*}
\varphi: H^1(Y_0;\mathbb{Z})&\rightarrow H^1(Y_1;\mathbb{Z})\\
\Phi:\mathrm{Spin}(\spin_0)&\rightarrow\mathrm{Spin}(\spin_1)
\end{align*}
such that the linear part of $\Phi$ is the reduction modulo $2$ of $\varphi$;
\item we have $\varphi^*(\cup^3_{Y_1})=\cup^3_{Y_0}$, and $\mu_{\spin_1}\circ\Phi-\mu_{\spin_0}$ is constant.
\end{itemize}
Then there is an isomorphism of graded $\Rin$-modules
\begin{equation*}
\HSb_*(Y_0,\spin_0)\equiv\HSb_*(Y_1,\spin_1),
\end{equation*}
up to an overall grading shift.
\end{thm}
The two main protagonists of the theorem, the Rokhlin map $\mu_{\spin}$ and the triple cup product $\cup^3_Y$ are in fact intimately connected.
\begin{prop}\label{cubic}
The Rokhlin map $\mu_{\spin}$ on $\mathrm{Spin}(\spin)$ is cubic, with its cubic part given by the reduction modulo $2$ of the triple cup product.
\end{prop}

We point out that while our main result shows that $\HSb_*(Y,\spin)$ is determined by rather explicit topological data, the actual computation of the group is in general quite laborious. We will provide several concrete examples in Section \ref{examples}.
\\
\par
The key inspiration for these results is the classic paper \cite{Ati}. The main difference is that while \cite{Ati} focuses on spin manifolds of dimension $8k+2$ (for which the spin Dirac operator is skew-adjoint), much of the content of the present paper can be generalized to spin manifolds of dimension $8k+3$ (for which the spin Dirac operator is quaternionic).
\par
The key ideas in the proof of our main result are the following. There is a natural involution on the torus of flat connections
\begin{equation*}
\mathbb{T}=H^1(Y;i\mathbb{R})/2\pi i H^1(Y;\mathbb{Z})
\end{equation*}
given by conjugation. This has exactly $2^{b_1(Y)}$ fixed points corresponding to the spin connections of the elements of $\mathrm{Spin}(\spin)$ or, equivalently, the flat connections with holonomy $\{\pm1\}$. Furthermore, the spinor bundle $S\rightarrow Y$ is quaternionic (with $j$ acting from the right), and the family of Dirac operators $\{D_B\}$ parametrized by $[B]\in\mathbb{T}$ satisfies
\begin{equation*}
D_B(\Psi\cdot j)=(D_{\bar{B}}\Psi)\cdot j.
\end{equation*}
The key point is that  $\HSb_*(Y,\spin)$ can be shown to depend \textit{only} on the homotopy class of this family of operators with involution, and the latter can be determined via a suitable $\K$-theoretic index theorem. As in the case of \cite{Ati} (see also \cite{AS5}), such a homotopy class cannot be recovered purely in cohomological terms. Fortunately enough, the information lost passing to cohomology can be fully recovered in terms of the $\mt$ spectral flow between the Dirac operators corresponding to the spin connections, which is in turn determined by the Rokhlin map $\mu_{\spin}$.
\\
\par
\textit{Plan of the paper.} In Section \ref{Ktheory}, after discussing a suitable $\K$-theoretic framework, we describe the index theorem relevant to our problem. In Section \ref{HS} we discuss the specialization of this index theorem to our case of interest, and prove Theorem \ref{main} and Proposition \ref{cubic}. In Section \ref{examples} we provide some explicit examples of computations.

\vspace{0.3cm}
\textit{Acknowledgements. }I am grateful to Tom Mrowka for many inspiring discussions and for his lectures about Atiyah's paper during his \textit{Riemann Surfaces} class in Spring 2013 at MIT. This work was supported by the the Shing-Shen Chern Membership Fund and the IAS Fund for Math.

\vspace{0.5cm}
\section{Quaternionic $\K$-theory and the index theorem for families}\label{Ktheory}

In this section we discuss the index theorem which is relevant to our situation. While we were not able to find in the literature the exact result we will need, the content of this section is quite standard and the proofs are straightforward adaptations of those appearing in the classical papers in the subject (to which we will refer for more details).
\\
\par
\textbf{Quaternionic $\K$-theory. }We start by discussing a simple variant of Atiyah's real $\K$-theory of spaces with a involution $\KR$ introduced in \cite{AtiR}. The key observation behind the following construction is that a real vector space can be thought as a complex vector space $E$ equipped with a complex antilinear map squaring to $1_E$, while a quaternionic vector space is a complex vector space $E$ equipped with a complex antilinear map squaring to $-1_E$. We will denote the latter by $j$. From this description we see that for any space $X$ there are natural product maps
\begin{align*}
\KO^i(X)\otimes \KO^j(X)&\rightarrow \KO^{i+j}(X)\\
\KO^i(X)\otimes \KSp^j(X)&\rightarrow \KSp^{i+j}(X)\numberthis \label{KOKR}\\ 
\KSp^i(X)\otimes \KSp^j(X)&\rightarrow \KO^{i+j}(X),
\end{align*}
given by tensor products.
\\
\par
Suppose now $X$ is a compact Hausdorff space equipped with an involution (denoted either by $\tau$ or $x\mapsto \bar{x}$). The key example is the $n$-dimensional torus $(\mathbb{R}/2\pi \mathbb{Z})^n$ with the involution induced by $x\mapsto -x$. We will denote this space with involution by $\mathbb{T}^n$ or, when the dimension is not important, simply by $\mathbb{T}$.
A \textit{real} (resp. \textit{quaternionic}) vector bundle $E$ is a bundle equipped with a complex antilinear map $j$ covering $\tau$ such that $j^2$ is $1_E$ (resp. $-1_E$). The Grothendieck group of real vector bundles was introduced in \cite{AtiR} and is denoted by $\KR(X)$. It is important to remark that it is \textit{not} the $\mathbb{Z}_2$-equivariant $\K$-theory of $X$.
\begin{defn}
Given a space with involution $X$, its \textit{quaternionic $\K$-theory} $\KQ(X)$ is defined to be the Grothendieck group of quaternionic vector bundles.
\end{defn}
There are natural maps
\begin{equation*}
\K(X)\leftarrow\KQ(X)\rightarrow \KSp(X^{\tau})
\end{equation*}
given respectively by forgetting the quaternionic structure and restricting to the fixed points of the involution. If $X_*$ is a space with basepoint $\ast$ (which we assume to be a fixed point for $\tau$), we can define the reduced quaternionic $K$-theory $\widetilde{\KQ}(X_*)$ as the kernel of the augmentation map to
\begin{equation*}
\KQ(X_*)\rightarrow \KQ(\ast).
\end{equation*}
As usual, we can then extend the definition of $\KQ$ to define a cohomology theory on spaces with or without basepoints via
\begin{align*}
\widetilde{\KQ}^{-i}(X_*)&=\widetilde{\KQ}(\Sigma^iX_*)\\
\KQ^{-i}(X)&=\widetilde{\KQ}(\Sigma^i (X\cup \ast)).
\end{align*}
Here by $\Sigma X_*$ we denote the reduced suspension
\begin{equation*}
[0,1]\times X_*/(\{0\}\times X_*\cup \{1\}\times X_*\cup [0,1]\times\{\ast\})
\end{equation*}
with the involution induced by $(t,x)\mapsto (t, \tau(x))$. We can also define the alternative suspension $\widetilde{\Sigma} X_*$ with the same underlying space but involution induced by $(t,x)\mapsto (1-t,\tau(x))$. The following is the key version of Bott periodicity for this version of $\K$-theory.
\begin{prop}\label{periodicity}
There are canonical isomorphisms
\begin{align*}
\KQ^{i+8}(X)&\cong \KQ^i(X)\\
\widetilde{\KQ}^i(\widetilde{\Sigma}^jX_*)&\cong \widetilde{\KQ}^{i+j}(X_*)
\end{align*}
\end{prop}
Hence, roughly speaking, suspension by $\tilde{\Sigma}$ is the inverse of suspension by $\Sigma$ on reduced $\KQ$-theory.
\begin{proof}
The proof readily using analogue of the products (\ref{KOKR}) given by
\begin{align*}
\KQ^i(X)\otimes {\KSp}^j(*)&\rightarrow \KR^{i+j}(X)\\
\KR^i(X)\otimes {\KSp}^j(*)&\rightarrow \KQ^{i+j}(X),
\end{align*}
and the periodicity theorems for $\KR$ and $\KSp$.
\end{proof}
For example, using the above we have
\begin{multline*}
\KQ^i(\mathbb{T}^1)=\KQ^{i-8}(\mathbb{T}^1)=\\=\widetilde{\KQ}^{i-1}(\Sigma^7(\mathbb{T}^1\cup \ast))=\widetilde{\KSp}^{i-1}(S^6\vee S^7)={\KSp}^{i-7}(\ast)\oplus{\KSp}^{i-8}(\ast),
\end{multline*}
where we used that $\mathbb{T}\wedge X$ is naturally identified with $\widetilde{\Sigma}X$. Recall the basic computation for $\KSp$-theory following from Bott periodicity
\begin{center}
\begin{tikzpicture}
\matrix (m) [matrix of math nodes,row sep=0.5em,column sep=1em,minimum width=2em]
{i& 0 & 1 &2&3&4&5&6&7\\
\KSp^{-i}(\ast) &\mathbb{Z}& 0&0&0&\mathbb{Z}&\mathbb{Z}/2\mathbb{Z}&\mathbb{Z}/2\mathbb{Z}&0\\
};
\draw(m-1-1.south west)--(m-1-9.south east);
\end{tikzpicture}
\end{center}
so that, for example
\begin{equation}\label{onedim}
\KQ^1(\mathbb{T}^1)=\mathbb{Z}/2\mathbb{Z}.
\end{equation}
More in general, we have the following.
\begin{lemma}\label{torus}We have the isomorphism
\begin{equation*}
\KQ^1(\mathbb{T}^n)=\mathbb{Z}_2^{a_n}\oplus \mathbb{Z}^{b_k}
\end{equation*}
where
\begin{equation*}
a_n=\bigoplus_{\substack{k\equiv 1,2 \text{ mod } 8\\ \quad1\leq  k\leq n}} {n\choose k}\qquad
b_n=\bigoplus_{\substack{k=3,7 \text{ mod } 8 \\ \quad1\leq  k\leq n}} {n\choose k}
\end{equation*}
\end{lemma}
\begin{proof}
From the discussion above, we see that
\begin{equation*}
\KQ^1(\mathbb{T}^n)=\widetilde{\KQ}(\Sigma^7(\mathbb{T}^n)).
\end{equation*}
Recall the basic identity (see for example \cite{Hat}, Proposition $4I.1$)
\begin{equation*}
\Sigma(X\times Y)= \Sigma X\vee \Sigma Y\vee\Sigma (X\wedge Y),
\end{equation*}
which also holds at the level of space with involutions. Applying this to $\mathbb{T}^n=\mathbb{T}^1\times\mathbb{T}^{n-1}$, by the periodicity theorem we obtain that
\begin{align*}
\widetilde{\KQ}^i(\mathbb{T}^{n})&= \widetilde{\KQ}^{i+1}(\Sigma\mathbb{T}^n)\\
&= \widetilde{\KQ}^{i+1}(\Sigma\mathbb{T}^1)\oplus \widetilde{\KQ}^{i+1}(\Sigma\mathbb{T}^{n-1})\oplus \widetilde{\KQ}^{i+1}(\Sigma(\mathbb{T}\wedge\mathbb{T}^{n-1}))\\
&= \widetilde{\KQ}^{i}(\mathbb{T}^1)\oplus \widetilde{\KQ}^{i}(\mathbb{T}^{n-1})\oplus \widetilde{\KQ}^{i+1}(\mathbb{T}^{n-1}).
\end{align*}
From this it inductively follows that
\begin{equation*}
\KQ^1(\mathbb{T}^n)=\bigoplus_{k=1}^n \widetilde{\KQ}^k(\mathbb{T})^{\oplus {n \choose k}},
\end{equation*}
and the result follows.
\end{proof}

\vspace{0.3cm}
\textbf{An index theorem for quaternionic families. }Consider a quaternionic bundle on a compact Riemannian manifold $E\rightarrow M$, with $j$ acting from the left. A \textit{quaternionic family} $\zeta$ is a continuous family $\{T_p\}_{p\in P}$ of first order elliptic self-adjoint operators
\begin{equation*}
T_p:C^{\infty}(E)\rightarrow C^{\infty}(E)
\end{equation*}
where $P$ is a space with involution $p\mapsto \bar{p}$ and we have
\begin{equation*}
T_p(s\cdot j)=(T_{\bar{p}}s)\cdot j\text{ for every }s\in C^{\infty}(E). 
\end{equation*}
To such an object we can associate a topological index and an analytical index as follows (see \cite{APS}, which treats the classical case without involutions, for more details). Let
\begin{equation*}
\pi:SM\rightarrow M
\end{equation*}
the projection from the unit cotangent bundle. For each $p\in P$, the symbol of $T_p$ defines a self-adjoint automorphisms $\sigma_p$ of the bundle $\pi^* E\rightarrow SM$. This defines a decomposition $\pi^*E=E^+_p\oplus E^-_p$ in positive and negative eigenspaces. In particular, we get a vector bundle
\begin{equation*}
E^+_{\zeta}=\bigcup_{p\in  P} E^+_p\rightarrow SM\times P
\end{equation*}
which is indeed quaternionic for the action of $j$ on $E$, hence a class $[E^+_P]\in \KQ(SM\times P)$. The \textit{symbol class} of $\zeta$ is defined to be the image of this class under the coboundary map
\begin{equation*}
\delta: \KQ^1(SM\times P)\rightarrow \KQ^1(TM\times P).
\end{equation*}
More concretely, this can be described as follows (see Lemma $3.1$ in \cite{APS}).  We say that a symbol is positive/negative if $E^{\mp}_{\zeta}=0$. Two self-adjoint symbols are said to be \textit{stably equivalent} if 
\begin{equation*}
\sigma\oplus \alpha\oplus \beta\approx \sigma'\oplus \alpha'\oplus \beta'
\end{equation*}
with $\alpha,\alpha'$ positive and $\beta,\beta'$ negative. Then the image of the coboundary map $\delta$ is naturally in bijection with stable equivalence classes of self-adjoint symbols. Finally, we define the \textit{topological index} to be the image of the symbol class under the index map
\begin{equation*}
\KQ^1(TM\times P)\rightarrow \KQ^1(P),
\end{equation*}
see \cite{AS4} for more details about the latter.
\par

Fix now a quaternionic separable Hilbert space $(H,j)$, and consider the space $\Op$ of complex linear Fredholm self-adjoint operators. This comes with a natural involution sending the operator $T$ to the operator
\begin{equation*}
v\mapsto -(T(v\cdot j))\cdot j,
\end{equation*}
the fixed points of which are exactly the quaternionic linear operators. With this involution, this is naturally a classifying space for $\KQ^1$, i.e.
\begin{equation}\label{classifying}
\KQ^1(P)=[P,\mathrm{Op}]_{\mathbb{Z}/2\mathbb{Z}}
\end{equation}
where on the right-hand side we consider the homotopy classes of equivariant maps. The proof of this fact follows quite closely the non-equivariant case. First of all, one identifies the space $\mathrm{Fred}_0$ of Fredholm operators on $H$ of index zero (with the same involution as above) as a classifying space for $\widetilde{\KQ}$ (see for example \cite{AtiK}). Then, one shows that there is an equivariant homotopy equivalence
\begin{equation}\label{equiv}
\mathrm{Op}\approx \widetilde{\Omega}\mathrm{Fred}_0
\end{equation}
where, for a space with involution $(X,\tau)$ we denote its $\widetilde{\Omega}X$ loop space (based at a fixed point of $\tau$) equipped with the involution sending a based loop $\gamma(t)$ for $t\in S^1$ to the loop $\tau(\gamma(\bar{t}))$. The operation $\widetilde{\Omega}$ is the adjoint of $\tilde{\Sigma}$. To show (\ref{equiv}), fix a quaternionic linear operator $T_0\in\mathrm{Op}$, and consider for $T\in\mathrm{Op}$ the loop $\gamma_T(t)$ in $\mathrm{Fred}_0$ given by
\begin{equation*}
\gamma_T(t)=\begin{cases}i \cos(t)+\sin(t)T,\quad t\in[0,\pi]\\
i \cos(t)-\sin(t)T_0,\quad t\in[\pi,2\pi],
\end{cases}
\end{equation*}
which we think as based at $t=3/2\pi$.
As (right) multiplication by $i$ is sent by conjugation to $-i$, if we consider the involution sending $t$ to $\pi-t$ $\mathrm{mod}$ $2\pi$ (and $t=3/2\pi$ as the basepoint) we obtain an equivariant map
\begin{equation*}
\mathrm{Op}\rightarrow \widetilde{\Omega}\mathrm{Fred}_0,
\end{equation*}
where the loop space is based at $T_0$. This maps is essentially the one introduced in \cite{ASskew}, hence it is a homotopy equivalence.
\par
By (\ref{classifying}), the quaternionic family of operators $\zeta$ parametrized by the space with involution $P$ determines (after performing a standard trick to make them Fredholm) a class in $\KQ^1(P)$ called the \textit{analytical index}. We then have the following.
\begin{prop}
For a quaternionic family of operators $\zeta$ parametrized by $P$, the topological index coincides with the analytical index in $\KQ^1(P)$.
\end{prop}
The proof of the result follows very closely that of the classic index theorem for families of self-adjoint operators (see \cite{APS}, to which we refer for details). For example, using a loop construction as above one obtains a quaternionic family of elliptic symbols giving rise to an element in $\KQ^0(S^1\times X\times TY)$, and similarly the family of operators induces a class in $\KQ^0(S^1\times X)$. From this, one reduces to the usual index theorem (cfr. \cite{AtiR}).

\vspace{0.5cm}

\section{Applications to monopole Floer homology}\label{HS}

We now discuss the implications of the index theorem discussed in the previous section for monopole Floer homology. As a technical remark, in the present section we will be working in the setting of unbounded self-adjoint operators $S_*(H:H_1)$ on a Hilbert space $H$ with a suitable dense subspace $H_1$ introduced in Chapter $33$ of \cite{KM}. This differs from that employed in the previous section (where we considered the space $\mathrm{Op}$) and in the classic papers on index theory (as for example \cite{APS}), but the main results are readily adapted as in \cite{KM}. The reader not too interested about these technical aspects should think of $S(H:H_1)$ as a suitable generalization of the space of compact self-adjoint perturbations of Dirac operators (in which case $H=L^2(Y;S)$ and $H_1=L^2_1(Y;S)$).

\vspace{0.3cm}
\textbf{Families of Dirac operators. }While the results we will discuss hold more in general for $(8k+3)$-dimensional spin manifolds, we will focus on the case of of three-manifolds. Let $Y$ be a three-manifold equipped with a self-conjugate spin$^c$ structure $\spin$, and consider the family of Dirac operators $\{D_B\}_{[B]\in\mathbb{T}}$ parametrized by the $b_1(Y)$-dimensional torus $\mathbb{T}$ of flat connections. The torus $\mathbb{T}$ comes with the natural involution by conjugation (or, equivalently, $x\mapsto -x$), which has exactly $2^{b_1(Y)}$ fixed points corresponding to the spin connections of the elements of $\mathrm{Spin}(\spin)$.
\par
We start by discussing a key example. Suppose $b_1(Y)=1$. We know from (\ref{onedim}) that $\KQ^1(\mathbb{T}^1)=\mathbb{Z}/2\mathbb{Z}$, so that there are exactly two homotopy classes of families of quaternionic operators. These can be described explicity as follows. Given any family over $\mathbb{T}^1$, the operators at $0$ and $\pi$ are quaternionic, so that in particular their kernel is even dimensional over $\mathbb{C}$. This implies that the $\mt$ spectral flow between the two operators is well-defined. Indeed, both cases can be realized: in the case of $S^2\times S^1$ we have even spectral flow, while in the case of the zero surgery on the trefoil we have odd spectral flow (see Chapter $4$ of \cite{Lin}). Hence the $\mt$ spectral flow is a complete invariant for families.
With this in mind, we prove the following.
\begin{prop}\label{key}
Let $\{D_B\}_{[B]\in\mathbb{T}}$ be the family of Dirac operators for $(Y,\spin)$. Then its index in $\KQ^1(\mathbb{T})$ is determined by the triple cup product of $Y$ together with the $\mt$ spectral flow between the Dirac operators corresponding to the spin connections of the elements of $\mathrm{Spin}(\spin)$.
\end{prop}
\begin{proof}
Consider the map
\begin{equation}\label{forget}
\KQ^1(\mathbb{T})\rightarrow \K^1(\mathbb{T})
\end{equation}
which forgets the involution. This map is injective after tensoring with $\mathbb{Q}$. This is clear from the computation of the first group (see Lemma \ref{torus}), and the fact that statement is true for the natural map $\KSp^i(*)\rightarrow \K^i(*)$. Of course, this map simply sends the family to its index (as a family of self-adjoint operator), which is in turn determined via the Chern character
\begin{equation*}
\mathrm{ch}:K^1(\mathbb{T})\rightarrow H^{\mathrm{odd}}(\mathbb{T};\mathbb{Q})
\end{equation*}
by the triple cup product, as it follows from the usual index theorem for families (see Lemma $35.1.2$ of \cite{KM} for an explicit computation). An important point here is that $\K^1(\mathbb{T})$ is torsion-free, so no information is lost by taking the Chern character. Hence we only need to understand the torsion of $\KQ^1(\mathbb{T})$, which is according to Lemma \ref{torus} a suitable direct sum of $\mathbb{Z}/2\mathbb{Z}$s.
\par
Pick a basis $\x_1,\dots, \x_n$ of $H^1(Y,\mathbb{Z})$. This determines a splitting in one-dimensional factors $\mathbb{T}=\mathbb{T}_1\times \cdots \times \mathbb{T}_n$, and (choosing a base spin structure $\s_0$) we can also use it to identify the set $\mathrm{Spin}(\spin)$ with subsets of $\{1,\dots,n\}$. From the computation in Lemma \ref{torus} we then know that the summands $\mathbb{Z}/2\mathbb{Z}$ to $\KQ^1(\mathbb{T})$ arise from wedge summands of the form
\begin{equation*}
\mathbb{T}_{i_1}\wedge\dots\wedge \mathbb{T}_{i_k}
\end{equation*}
where $k=1$ or $2$ modulo $8$. Consider now a standard equivariant loop $\gamma$ whose fixed points are precisely $\s_0$ and the spin structure corresponding to $\{i_1,\dots, i_k\}$. Then the suspension of the inclusion $\gamma\hookrightarrow\mathbb{T}$ factors through the suspension of
\begin{equation*}
\gamma\hookrightarrow\mathbb{T}_{i_1}\wedge\dots\wedge \mathbb{T}_{i_k}.
\end{equation*}
The latter induces an isomorphism on $\KQ^1$: this follows for $k=2$ from the ring structure of ${\KSp}^*(*)$, and in general from the periodicity theorem (see Proposition \ref{periodicity}). Hence the component in each $\mathbb{Z}/2\mathbb{Z}$ summand can be interpreted in terms of the restriction of the family to a certain equivariant loop $\gamma$, which is in turn determined by a certain $\mt$ spectral flow by the discussion on the case $b_1=1$.
\end{proof}

\vspace{0.3cm}
\textbf{Coupled Morse homology. }Given a family of operators $L$ in $S_*(H:H_1)$ parametrized by a smooth manifold $Q$, in Chapter $33$ of \cite{KM} the authors construct the coupled Morse homology $\bar{H}_*(Q,L)$, which is a relatively graded module over $\ztwo[U]$ with $U$ of degree $-2$. We can assume in our context that the family has no spectral flow around loops in $Q$, so that the grading is absolute. We very quickly recall its definition, and refer the reader to \cite{KM} for more details. Fix a metric and choose a generic Morse function $f$ on $Q$. After a small perturbation of the family, we can assume that the following genericity assumption holds:
\begin{enumerate}[($\ast$)]
\item for all critical points $q$ of $f$, the operator $L_q$ has no kernel and simple spectrum.
\end{enumerate}
The chain complex for coupled Morse homology $\bar{C}_*(Q,L)$ is generated over $\ztwo$ by the projectivizations of the eigenspaces of the operators at the critical points (which are all one-dimensional by ($\ast$)). We then look at equivalence classes of pairs of paths $(\gamma(t),\phi(t))$ where $\gamma$ is a Morse trajectory for $f$ and $\phi$ is a path in the unit sphere of $H$ satisfying a given differential equation. The differential counts the number of these paths in zero dimensional moduli spaces.
The two key properties of this construction are the following:
\begin{itemize}
\item The coupled Morse homology $\bar{H}_*(Q,L)$ only depends on the homotopy class of the family $L$, hence on the corresponding element in $\K^1(Q)=[Q,S_*(H:H_1)]$.
\item If $(Y,\spin)$ is a three-manifold equipped with a torsion spin$^c$ structure, the result of the construction applied to the family of Dirac operators parametrized by flat connections $\{D_B\}_{[B]\in\mathbb{T}}$ is $\HMb_*(Y,\spin)$.
\end{itemize}

An analogous construction can be performed if the manifold $Q$ comes with an involution $\tau$ and the family of operators $L$ is quaternionic for this involution. We will assume that the fixed points of $\tau$ are isolated, so that it is locally modeled on $x\mapsto -x$. The main complication is that now the operators at the fixed points of the action are quaternionic-linear, so the transversality assumption $(*)$ cannot be achieved (respecting the involution) as the eigenspaces will always be even-dimensional over $\mathbb{C}$. The problem can be solved, as in \cite{Lin}, by allowing Morse-Bott singularities of a very specific kind. Indeed, generically the operators at the fixed points of $\tau$ will have no kernel and two-dimensional eigenspaces: each of these will give rise, after projectivization, to a copy of $S^2$ on which the involution $j$ acts as the antipodal map. With this in mind, the construction of $\Pin$-monopole Floer from \cite{Lin} carries over without significant differences (see also \cite{Lin3} for an introduction): indeed, as we do not have to deal with boundary obstructedness phoenomena, the technical details are significantly easier in this case. The output is a version of the chain complex $\bar{C}_*(Q,L)$ whose homology is $\bar{H}_*(Q,L)$ which is equipped with a natural chain involution. We then define $\bar{H}^{\tau}_*(Q,L)$ to be the homology of the invariant subcomplex. This is naturally a module over $\Rin$. We record the main features of this construction in the following result, whose proofs follows along the lines of the results cointained in \cite{Lin} (and in fact much simpler).
\begin{prop}For any quaternionic family of operators $L$ in $S_*(H:H_1)$ on $(Q,\tau)$ there is a well-defined equivariant coupled Morse homology group $\bar{H}^{\tau}_*(Q,L)$, which is an absolutely graded module over $\Rin$. The following properties hold:
\begin{itemize}
\item $\bar{H}^{\tau}_*(Q,L)$ only depends on the homotopy class of the quaternionic family, hence on the corresponding element in $\KQ^1(Q)=[Q,S_*(H:H_1)]_{\mathbb{Z}/2\mathbb{Z}}$;
\item If $(Y,\spin)$ is a three-manifold equipped with a self-conjugate spin$^c$ structure, the result of the construction applied to the family of Dirac operators parametrized by flat connections $\{D_B\}_{[B]\in\mathbb{T}}$ is $\HSb_*(Y,\spin)$.
\end{itemize}
\end{prop}

Putting the pieces together, we can finally prove the main result of the present paper.
\begin{proof}[Proof of Theorem \ref{main}]
Using the equivariant coupled Morse homology introduced above, we see that $\HSb_*(Y,\spin)$ only depends on the homotopy class of the quaternionic family of operators $\{D_B\}_{[B]\in\mathbb{T}}$, hence by Proposition \ref{key} only on the triple cup product on $Y$ and the $\mt$ spectral flow between the Dirac operators corresponding to elements of $\mathrm{Spin}(\spin)$. So, we only need to show that the latter are determined by the Rokhlin invariants of the spin structures. This can be seen by looking at the absolute gradings in the chain complex (as introduced in Chapter $28$ of \cite{KM}). Choose a standard equivariant Morse function $f$ on $\mathbb{T}$, so that its $2^{b_1(Y)}$ critical points correspond to the elements of $\mathrm{Spin}(\spin)$. Given two spin structures $\s_0$ and $\s_1$ with corresponding spin connections $B_0$ and $B_1$, we have that the zero-dimensional chains in a stable critical submanifold $C_i$ over $(B_i,0)$ have relative grading
\begin{equation*}
\mathrm{ind}_f(\s_0)-\mathrm{ind}_f(\s_1)-2\mathrm{sf}(D_{B_0},D_{B_1})\text{ mod }4.
\end{equation*}
On the other hand, zero-dimensional chains in $C_i$ have absolute grading $-\sigma(W_i)/4+\mathrm{ind}_f(\s_i)$ modulo $4$, where $W_i$ is any manifold whose boundary is $Y$ on which $\s_i$ extends. For this computation we exploit the fact that on a spin four-manifold the Dirac operator is quaternionic, so that its (real) index is divisible by $4$ (see also Chapter $4$ of \cite{Lin}). Comparing this with the formula above, we see
that the $\mt$ spectral flow between $D_{B_0}$ and $D_{B_1}$ is exactly the difference between the Rokhlin invariants of $\s_0$ and $\s_1$.
\end{proof}

\vspace{0.3cm}
\textbf{Rokhlin invariants as a cubic map.} The rest of this section is devoted to understand the structure of the Rokhlin invariants of a pair $(Y,\spin)$, and in particular to prove Proposition \ref{cubic}. First, recall some features of spin manifolds in low-dimensions (see for example \cite{Kir}). In dimension one, the spin cobordism group is $\Omega_1^{\mathrm{Spin}}=\mathbb{Z}/2\mathbb{Z}$, the generator being the trivial double cover of the circle (see Figure \ref{spincircle}). We will refer to this as the \textit{Lie structure}, and denote it by $\s_{\mathrm{Lie}}$. In dimension two, $\Omega_2^{\mathrm{Spin}}=\mathbb{Z}/2\mathbb{Z}$, and the spin cobordism class is determined by the Arf invariant of $(\Sigma,\s)$. More precisely, by restricting to loops the spin structure $\s$ determines a map
\begin{equation*}
q:H_1(\Sigma,\mathbb{Z})\otimes \ztwo\rightarrow \Omega_1^{\mathrm{Spin}}=\mathbb{Z}/2\mathbb{Z}
\end{equation*}
which is a quadratic refinement of the intersection product. The Arf invariant of this quadratic form in $\mathbb{Z}/2\mathbb{Z}$ is then the spin cobordism class of $(\Sigma,\s)$.

\begin{figure}
  \centering
\def\svgwidth{0.8\textwidth}
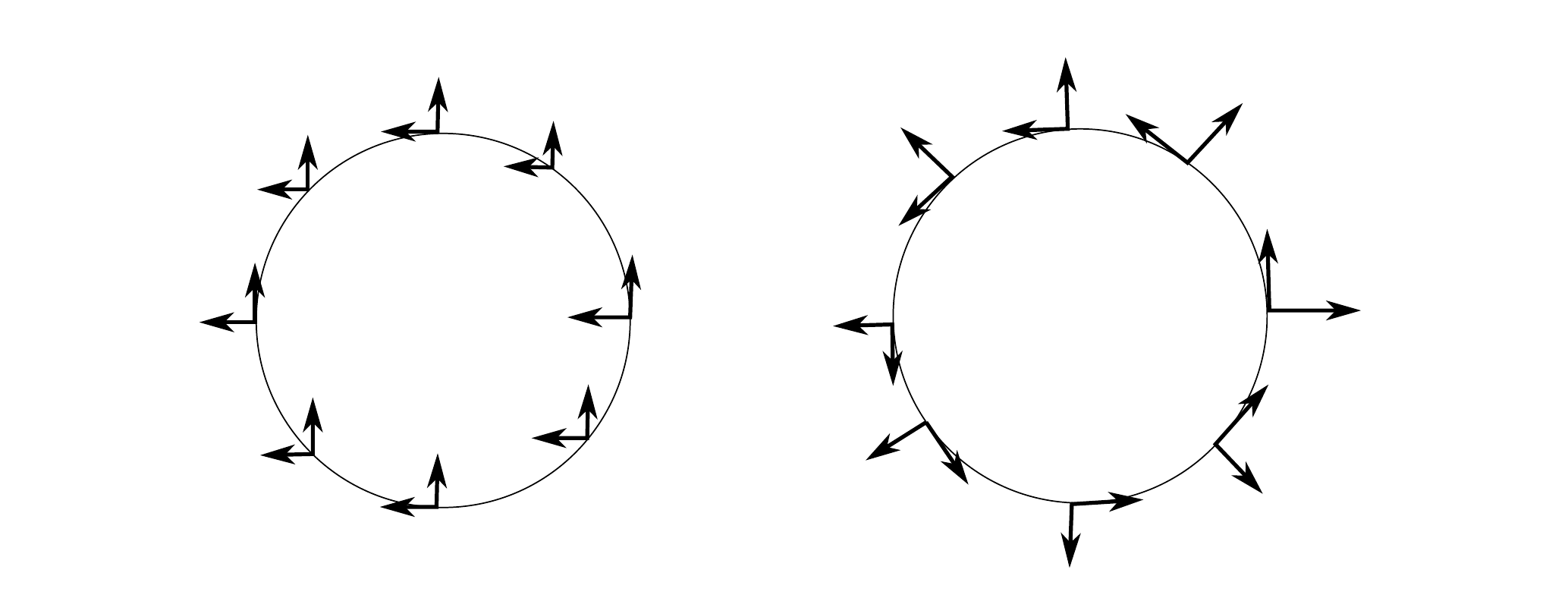
    \caption{The two spin structures on the circle. The one on the left corresponds to the \textit{non}-trivial double cover of the circle, and extends to the disk. The one of the right is the Lie structure, it corresponds to the trivial double cover of the circle, and does not extend.}
    \label{spincircle}
\end{figure} 

Consider now the three-manifold $Y=S^1\times S^1\times S^1$. This has $8$ spin structures, $7$ with Rokhlin invariant $0$ and one with Rokhlin invariant $1$. All the spin structures are all products of spin structures on the circle. If at least one of the factors is not the Lie structure, then it extends to a manifold diffeomorphic to $S^1\times S^1\times D^2$, which has signature zero. On the other hand, the product of the Lie structures (which we denote by $\s_0=\s_{\mathrm{Lie}}\times \s_{\mathrm{Lie}}\times \s_{\mathrm{Lie}}$) extends to a manifold with signature $-8$, namely the complement of a regular fiber of the elliptic fibration $E(1)\rightarrow \mathbb{C}P^1$ (see \cite{Kir}). With this in mind, we can proceed on the proof of Proposition \ref{cubic}.

\begin{proof}[Proof of Proposition \ref{cubic}]
Consider two spin structures $\s$ and $\s'$ in $\mathrm{Spin}(\spin)$, so that they differ by an element $\x\in H^1(Y;\mathbb{Z})\otimes \ztwo$. Let $\Sigma\subset Y$ be an oriented surface Poincar\'e dual to $\x$. We claim that the difference of their Rokhlin invariants satisfies
\begin{equation*}
\mu(\s)-\mu(\s')=[\s\lvert_{\Sigma}]\in \Omega_2^{\mathrm{Spin}}=\mathbb{Z}/2\mathbb{Z}.
\end{equation*}
To show this, we construct a suitable spin cobordism between $(Y,\s)$ and $(Y,\s')$ as follows. Consider the manifold $Y\times [-1,1]$, with fixed spin structures $\s$ and $\s'$ at the boundary. Then $\Sigma\times\{0\}$ is a characteristic surface (in a relative sense): there is a spin structure on its complement restricting to $\s$ and $\s'$ at the boundaries and which induces the non-trivial element in $\Omega_1^{\mathrm{Spin}}$ on the unit circle of a normal fiber of $\Sigma$. The proof of this is a direct generalization of the closed case, see \cite{Kir}. Consider now $\partial\mathrm{nbhd}(\Sigma\times\{0\})$, which is naturally identified with $S^1\times \Sigma$. This has the induced spin structure $\s_{\mathrm{Lie}}\times \s\lvert_{\Sigma}$. Hence to find a spin cobordism from $(Y,\s)$ to $(Y,\s')$ it suffices to find a spin manifold whose boundary is $S^1\times \Sigma$ on which this spin structure extends, and glue it in. The example of the three-torus discussed above (which is the case $S^1\times T^2$) readily implies that if $\s\lvert_{\Sigma}$ is trivial, one can find an extension to a manifold with $\sigma=0$ modulo $16$, and if $\s\lvert_{\Sigma}$ is not, the one can find an extension to a manifold with $\sigma=8$ modulo $16$. This is because, from the properties of the Arf invariant, writing $\Sigma=\hash^g T^2$, $\s\lvert_{\Sigma}$ bounds if and only if the number of restriction to the summands $\s\lvert_{T^2}$ that do not bound is even.
\par
To conclude the proof, we need to show that the map
\begin{align*}
H^1(Y;\mathbb{Z})\otimes \ztwo&\rightarrow \mathbb{Z}/2\mathbb{Z}\\
\x&\mapsto [\s\lvert_{\mathrm{PD}(\x)}]
\end{align*}
is cubic. To see this, fix a basis $\x_1,\dots, \x_n$, with transverse dual surfaces $\Sigma_1,\dots, \Sigma_n$. Then if $\x=\sum \lambda_i \x_i$, and $\Sigma$ is dual to $\x$, we have
\begin{equation*}
[\s\lvert_{\Sigma}]=\sum \lambda_i [\s\lvert_{\Sigma_i}]+\sum \lambda_i\lambda_j[\s\lvert_{\Sigma_i\cap \Sigma_j}]+\sum \lambda_i\lambda_j\lambda_k[\s\lvert_{\Sigma_i\cap \Sigma_j\cap\Sigma_k}],
\end{equation*}
as it can be seen directly from a cut and paste argument (see for example \cite{Kir}). Finally, the term $[\s\lvert_{\Sigma_i\cap \Sigma_j\cap\Sigma_k}]$ is clearly the triple cup product $\langle \x_i\cup \x_j\cup \x_k,[Y]\rangle$ modulo two.
\end{proof}

Indeed, the proof shows that not only the function is cubic (with cubic part determined by triple cup product), but also that the linear and quadratic coefficients can be determined very explicitly in terms of embedded surfaces. For example, in the case of the three-torus discussed above, fix $\s_0$ as the base spin structure. Using this to identify the set of spin structures with $H^1(Y;\mathbb{Z})\otimes \ztwo$, the map $\mu$ is $1$ on any non-zero element. Denoting by $\x_i$ the generator of the circle in the $i$th factor, we see that this is the cubic map
\begin{equation*}
\sum \lambda_i \x_i\mapsto 1+\lambda_1+\lambda_2+\lambda_3+\lambda_1\lambda_2+\lambda_2\lambda_3+\lambda_1\lambda_3+\lambda_1\lambda_2\lambda_3.
\end{equation*}
The coefficients can be interpreted topologically as follows. First consider a surface $\Sigma_i$ Poincar\'e dual to $\x_i$. Then $\s_0\lvert_{\Sigma_i}$ is a product of Lie structures, hence it has Arf invariant $1$, which corresponds to the coefficient of $\lambda_i$. Furthermore, for $i\neq j$ the intersection $\Sigma_i\cap \Sigma_j$ has induced the Lie structure, hence it is $1\in \Omega_1^{\mathrm{Spin}}$. This corresponds to the coefficient of $\lambda_i\lambda_j$.

\vspace{0.5cm}

\section{Examples}\label{examples}
We now discuss some explicit computations of the group $\HSb_*$ in terms of the topological data arising from Theorem \ref{main}. We start by recalling the simplest cases of manifolds with $b_1=0$ and $1$, which were worked out in Chapter $4$ of \cite{Lin}. Define
\begin{equation*}
\tilde{\Rin}=\ztwo[V^{-1},V,Q]/Q^3,
\end{equation*}
which is naturally a module over $\Rin$, and
\begin{equation*}
\mathcal{I}=\ztwo[V^{-1},V]\oplus \ztwo[V^{-1},V]\langle-1\rangle
\end{equation*}
where the action of $Q$ is an isomorphism from the first tower onto the second. Here, given a graded module $M$, we denote by $M\langle d\rangle $ the module obtained by shifting the degrees up by $d$, i.e. $M\langle d\rangle_i=M_{d-i}$. When $b_1=0$, the final result is simply the direct sum of the homologies of the critical submanifolds, hence up to a total grading shift
\begin{equation*}
\HSb_*(Y,\spin)=\tilde{\Rin}.
\end{equation*}
When $b_1=1$, as mentioned in the previous section, there are two cases corresponding to the two elements of $\KQ^1(\mathbb{T})=\mathbb{Z}/2\mathbb{Z}$. We can pick a standard equivariant function on $\mathbb{T}$ with exactly two critical points. We denote by $\s$ and $\s'$ its maximum and minimum. If the two spin structures have the same Rokhlin invariant (so that there is no $\mt$ spectral flow between the spin two Dirac operators), for each critical submanifold $C$ over $\s$ there is a critical submanifold $C'$ over $\s'$ lying in degree one less; furthermore the moduli space of trajectories between $C$ and $C'$ consists of two copies of $\mathbb{C}P^1$, each mapping diffeomorphically onto the images under the evaluation map, so that
\begin{equation*}
\HSb_*(Y,\spin)=\tilde{\Rin}\otimes H^1(S^1;\ztwo).
\end{equation*}
In the case they have different Rokhlin invariants, for each critical submanifold $C$ over $\s$ there is a critical submanifold $C'$ over $\s'$ lying in degree one more; and the moduli space of trajectories between them consist of two points, inducing multiplication by $Q^2$ in homology. Hence, the final result is 
\begin{equation*}
\HSb_*(Y,\spin)=\mathcal{I}\oplus\mathcal{I}\langle2\rangle.
\end{equation*}
The first case is realized for example by manifolds obtained by zero surgery on a knot $K\subset S^3$ with $\mathrm{Arf}(K)$ zero, while the latter happens when $\mathrm{Arf}(K)$ is one.
\vspace{0.3cm}

With this computations in mind, we can prove a general result as follows. Here, we fix a basis $\x_1,\dots, \x_n$ of $H^1(Y;\mathbb{Z})\otimes \ztwo$, and a base spin structure $\s_0$. Given a subset $I\subset\{1,\dots,n\}$, we denote by $\s_I$ the spin structure $\s_0+\sum \lambda_i \x_i$, where $\lambda_i=1$ iff $i\in I$. We denote by $|I|$ the cardinality of $I$. 
\begin{prop}\label{spectralseq}
There exists a spectral sequence converging to $\HSb_*(Y,\spin)$ whose $E^1$ page is
\begin{equation*}
\sum_{I}\tilde{\Rin}\langle -2\mu(\s_I)+|I|)\rangle
\end{equation*}
so that the differential $d^1$ has a non-zero component
\begin{equation*}
\tilde{\Rin}\langle -2\mu(\s_I)+|I|\rangle\rightarrow\tilde{\Rin}\langle -2\mu(\s_{I'})+|I'|\rangle
\end{equation*}
if and only $I'\subset I$, $|I'|=|I|-1$ and $\mu(\s_I)\neq\mu(\s_{I'})$, in which case it is (up to grading shift) multiplication by $Q^2$.
\end{prop}

\begin{proof}To compute the Floer homology group we can proceed as in \cite{Lin} and use an equivariant Morse function
\begin{equation*}
f:\mathbb{T}\rightarrow \mathbb{R}
\end{equation*}
to perturb the equations. We can choose $f$ to be a standard Morse function on the torus, whose critical points correspond to the $2^{b_1(Y)}$ fixed points of the involution, hence to the set of spin structures $\mathrm{Spin}(\spin)$. Using our basis, we can also choose the function so that the index of the critical point corresponding to $\s_I$ is $|I|$. We can filter the Floer chain complex according to the index of the corresponding critical point of $f$. The $E^1$ page is then the direct sum of the homology of the critical submanifolds, so that one obtains the first part of the statement after recalling that the index of the critical point contributes to the grading shift; the statement regarding the differentials $d_1$ follows from our computations in the case $b_1=1$. \end{proof}

This corollary does not provide an explicit computation for the group, but just a spectral sequence for which we know explicitly the $E^2$-page. Indeed, also the explicit general computation of \cite{KM} assumes rational coefficients and exploits the formality of the de Rham cohomology of $\mathbb{T}$. In our setting, unlike the case of \cite{KM}, the differential $d_2$ is generally non-zero, as we will see explicitly in an example.

\vspace{0.3cm}
\textbf{$\Pin$-standard manifolds. }We say that a three-manifold $Y$ equipped with a self-conjugate spin$^c$ structure $\spin$ is $\Pin$-\textit{standard} if the triple cup product of $Y$ vanishes and the spin structures in $\mathrm{Spin}(\spin)$ all have the same Rokhlin invariant. We claim that in this case we have
\begin{equation*}
\HSb_*(Y,\spin)=\tilde{\Rin}\otimes H^1(\mathbb{T};\ztwo).
\end{equation*}
Indeed, our main theorem implies that, up to grading shift,
\begin{equation*}
\HSb_*(Y,\spin)=\HSb_*(\hash^{b_1(Y)}S^2\times S^1,\spin_0)
\end{equation*}
where $\spin_0$ is the unique torsion spin$^c$ structure. The latter can be computed for example using the connected sum spectral sequence (see \cite{Lin2}). Indeed, we know
\begin{equation*}
\HSb_*(S^2\times S^1,\spin_0)=\tilde{\Rin}\otimes H^1(S^1;\ztwo)
\end{equation*}
and as this is free over $\tilde{\Rin}$, the invariant of the connected sum is simply the tensor product over $\Rin$ of the invariants (as the spectral sequence collapses at the $E^2$-page).
\vspace{0.3cm}

\textbf{Manifolds with $b_2=2$. }In this case the triple cup product vanishes, so that the invariant is determined by the Rokhlin invariants. If all of them coincide, then the manifold is $\Pin$-standard so that the result discussed above holds.  On the other hand, the cup product of two basis elements of $H^1(Y;\mathbb{Z})$ has to be zero by Poincar\'e duality. Hence the cubic form from Theorem \ref{cubic} has to be linear, so that the four Rokhlin invariants coincide in pairs. In particular, in light of Theorem \ref{main}, we can compute $\HSb_*(Y,\spin)$ as the invariant for the manifold obtained by zero surgeries on each component of a split link, one component being a trefoil and one component being unknotted. For this case
\begin{equation*}
\HSb_*(Y,\spin)=\left(\mathcal{I}\oplus\mathcal{I}\langle2\rangle\right)\otimes H^1(S^1;\ztwo),
\end{equation*}
as it follows by looking at the connected sum spectral sequence.

\vspace{0.3cm}

\textbf{Manifolds with $b_2=3$. }There are several cases to discuss. First of all, if we pick a basis $\x_1,\x_2,\x_3\in H^1(Y;\mathbb{Z})$, the value
\begin{equation*}
\langle \x_1\cup \x_2\cup \x_3,[Y]\rangle\in \mathbb{Z}
\end{equation*}
is well-defined (up to sign). With a little abuse of terminology, we will refer to $m$ as the triple cup product of $Y$. Recall that examples of three-manifolds with triple cup product $m$ can be provided by the construction in \cite{RS} by doing surgery on a band sum of $m$ copies of the Borromean rings, see Figure \ref{massey}.
\par
First of all, we consider the case in which the triple cup product is even. By Poincar\'e duality, the cup product on $H^1(Y;\mathbb{Z})$ vanishes modulo two. As above, this implies that the Rokhlin function is linear. Hence either all Rokhlin invariants concide (in which case the manifold is $\Pin$-standard), or exactly half of them take one value. We can compute the homology in the latter case as follows. We can write the torus of flat connections as $\mathbb{T}^1\times \mathbb{T}^2$ in such a way that the spin structures in $\{0\}\times \mathbb{T}^2$ and $\{\pi\}\times \mathbb{T}^2$ all have the same Rokhlin invariant. We can consider the two step filtration coming from the value in the component $\mathbb{T}^1$. The $E^1$ page is the direct sum of the equivariant coupled Morse homologies of the families parametrized by $\{0\}\times \mathbb{T}^2$ and $\{\pi\}\times \mathbb{T}^2$, which are $\Pin$-standard. In particular, we have
\begin{equation*}
E^1= \tilde{\Rin}\otimes\left((H^1(\mathbb{T}^2;\ztwo))\oplus (H^1(\mathbb{T}^2;\ztwo))\langle 1\rangle\right).
\end{equation*}
Here the shift of the second summand comes from the difference of the Rokhlin invariants. Furthermore, each summand $H^1(\mathbb{T}^2;\ztwo)$ has a filtration coming from the index of the Morse function on $\mathbb{T}^2$.  The $d_1$ differential maps the first summand to the second, and it also lowers the filtration level on $H^1(\mathbb{T}^2;\ztwo)$. Using the description of the moduli spaces in the case $b_1=1$, the filtration preserving component is readily computed to be multiplication by $Q^2$ (up to grading shift). Hence we have
\begin{equation*}
\HSb_*(Y,\spin)=(\mathcal{I}\oplus\mathcal{I}\langle2\rangle)\otimes (H^1(\mathbb{T}^2;\ztwo))
\end{equation*}

In the case the triple cup product is odd, there are again two cases (as it can be shown by a direct inspection): either $7$ spin structures attain one value and the remaining one a different one (as in the case of the three-torus), or $5$ spin structures attain one value and $3$ attain the other. The latter case can be realized from the general example in Figure \ref{massey} by tying a knot of Arf invariant $1$ in one of the components. We already see a difference with the even case in usual monopole Floer homology: as shown in \cite{KM}, in the odd case $\HMb_*(Y,\spin)$ has rank $3$ in each degree, rather than $4$. In the first of the two possible cases, we will show that
\begin{equation*}
\HSb_*(Y,\spin)=\left(H^1(\mathbb{T}^3)\oplus H^2(\mathbb{T}^3)\right)\otimes \tilde{\Rin}.
\end{equation*}
An analogous computation was provided in different terms in the case of the three-torus in \cite{Lin}.
To see this, first recall that the Gysin exact triangle
\begin{center}\label{triangle}
\begin{tikzpicture}
\matrix (m) [matrix of math nodes,row sep=1em,column sep=1em,minimum width=2em]
  {
  \HSb_{*}(Y,\spin) && \HSb_{*}(Y,\spin)\\
  &\HMb_{*}(Y,\spin) &\\};
  \path[-stealth]
  (m-1-1) edge node [above]{$\cdot Q$} (m-1-3)
  (m-2-2) edge node [left]{} (m-1-1)
  (m-1-3) edge node [right]{} (m-2-2)  
  ;
\end{tikzpicture}
\end{center}
implies that (if we think of $\HSb_*$ as a $\ztwo[Q]/Q^3$-module) each cyclic summand corresponds to a rank two subgroup $\ztwo\oplus\ztwo$ of $\HMb_*$). Furthermore, if this summand is isomorphic to $\ztwo[Q]/Q^i$, then the generators of this subgroup differ in degree by $i$.
\par
We look at the spectral sequence from Proposition \ref{spectralseq}. The $E^1$ page is given by
\begin{center}
\begin{tikzpicture}
\matrix (m) [matrix of math nodes,row sep=0.5em,column sep=0.5em,minimum width=2em]
{\ztwo& &&\\
\ztwo &\ztwo^3&&\ztwo\\
\ztwo&\ztwo^3&\ztwo^3&\ztwo\\
&\ztwo^3&\ztwo^3&\ztwo\\
&&\ztwo^3&\\
};
\end{tikzpicture}
\end{center}
where each column repeats $4$-periodically. The groups in the $i$th column correspond to the critical points of index $3-i$. The spin structure corresponding to the last column has different Rokhlin invariant (so that it is shifted in degree by $2$). We will use the convention that between two consecutive groups in the same column the map $Q$ has the highest possible rank.

The differential $d_1$ goes from one column to the one on its right, and the $E^2$ page is given by
\begin{center}
\begin{tikzpicture}
\matrix (m) [matrix of math nodes,row sep=0.5em,column sep=0.5em,minimum width=2em]
{\ztwo& &&\\
\ztwo &\ztwo^3&&\ztwo\\
\ztwo&\ztwo^3&\ztwo^2&\ztwo\\
&\ztwo^3&\ztwo^3&\\
&&\ztwo^3&\\
};
\draw[->, dashed](m-2-2.east)--(m-3-4.west);
\draw[->](m-1-1.east)--(m-2-4.west);
\end{tikzpicture}
\end{center}
Because of the module structure, the only possible non-trivial differential $d_2$ is the dashed one. Furthermore, as $\HMb$ has rank $3$ in each degree, from the discussion on the Gysin sequence above the differential $d_3$ is forced to be the arrow drawn. Now, if $d_2$ is not zero, the final result is
\begin{center}
\begin{tikzpicture}
\matrix (m) [matrix of math nodes,row sep=0.5em,column sep=0.5em,minimum width=2em]
{\ztwo &\ztwo^2&&\\
\ztwo&\ztwo^3&\ztwo^2&\\
&\ztwo^3&\ztwo^3&\\
&&\ztwo^3&\\
};
\end{tikzpicture}
\end{center}
Of course, there are no possible extensions as $\ztwo[Q]/Q^3$-modules, so that this is indeed $\HSb_*$. On the other hand, this module requires $7$ generators over $\ztwo[Q]/Q^3$, so that one obtains a contradiction using the Gysin sequence. Hence $d_2$ vanishes and the $E^\infty$-page is
\begin{center}
\begin{tikzpicture}
\matrix (m) [matrix of math nodes,row sep=0.5em,column sep=1em,minimum width=2em]
{&\ztwo^3&&\\
\ztwo&\ztwo^3&\ztwo^2&\\
&\ztwo^3&\ztwo^3&\\
&&\ztwo^3&\\
};
\draw[->](m-2-1.east)--(m-3-3.west);
\end{tikzpicture}
\end{center}
Again this requires $7$ generators over $\ztwo[Q]/Q^3$, but there is a non-trivial extension, as shown by the arrow, and the result follows.
\par
Finally, in the case in which exactly $5$ spin structures have the same Rokhlin invariant, we have
\begin{equation*}
\HSb_*(Y,\spin)=(\mathrm{I}\oplus\mathrm{I}\langle2\rangle)\otimes H^1(\mathbb{T}^3;\ztwo)
\end{equation*}
by an analogous argument. We just point out that this is an example in which the differential $d_2$ of the spectral sequence is non-zero. Indeed, we can assume after a basis change that the Rokhlin map is
\begin{equation*}
\lambda_1\x_1+\lambda_2\x_2+\lambda_3\x_3\mapsto \lambda_1+\lambda_2+\lambda_3+\lambda_1\lambda_2\lambda_3
\end{equation*}
so that the $E^1$-page of the spectral sequence from Proposition \ref{spectralseq} looks like
\begin{center}
\begin{tikzpicture}
\matrix (m) [matrix of math nodes,row sep=0.5em,column sep=0.5em,minimum width=2em]
{&&&\ztwo\\
\ztwo& &\ztwo^3&\ztwo\\
\ztwo &\ztwo^3&\ztwo^3&\ztwo\\
\ztwo&\ztwo^3&\ztwo^3&\\
&\ztwo^3&&\\
&&&\\
};
\end{tikzpicture}
\end{center}
repeated four-periodically. So $E^2$ is
\begin{center}
\begin{tikzpicture}
\matrix (m) [matrix of math nodes,row sep=0.5em,column sep=0.5em,minimum width=2em]
{&&&\ztwo\\
\ztwo& &\ztwo^2&\ztwo\\
\ztwo &\ztwo&\ztwo^3&\\
\ztwo&\ztwo^3&\ztwo&\\
&\ztwo^3&&\\
&&&\\
};
\end{tikzpicture}
\end{center}
and, if we suppose that $d_2$ is zero, we see that also $d_3$ has to be zero (for degree reasons and because it is a map of $\Rin$-modules). On the other hand, this group cannot fit in the Gysin exact sequence with $\HMb_*$.
\begin{figure}
  \centering
\def\svgwidth{0.9\textwidth}
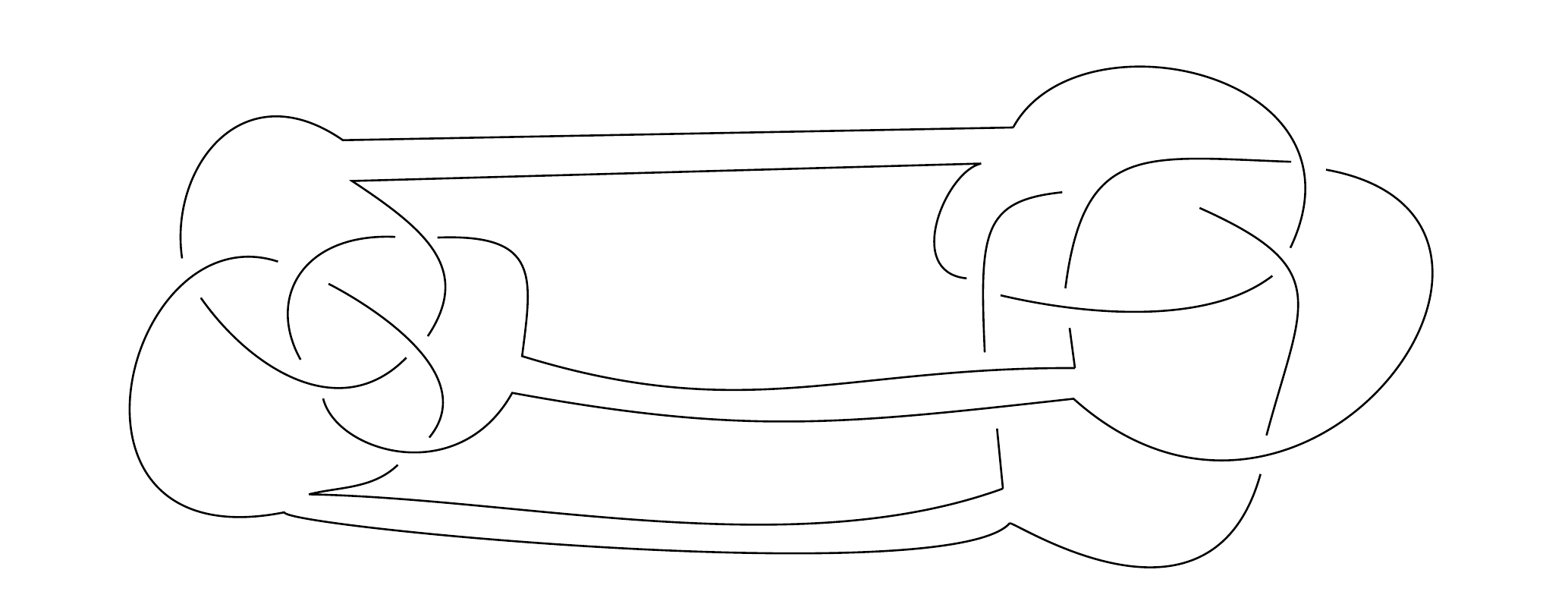
    \caption{The band sum of two copies of the Borromean rings. Taking the band sum of $n$ copies of the Borromean rings, and doing zero surgery on each of the components, one obtains a three manifold with $b_1=3$ and triple cup product $n$.}
    \label{massey}
\end{figure}

\vspace{0.5cm}

\bibliographystyle{alpha}
\bibliography{biblio}

\end{document}